\documentclass[11pt,english]{amsart}

\usepackage{amsmath,leftidx,amsthm}
\usepackage[all]{xy}
\usepackage{xspace,amsthm}
\usepackage[psamsfonts]{amssymb}
\usepackage[latin1]{inputenc}
\usepackage{graphicx,color,mathrsfs}
\usepackage{hyperref,fancyhdr,appendix}
\usepackage{xcolor}

\newtheorem*{example}{\bf Example}

\newtheorem*{remark}{\bf Remark}

\newtheorem{theorem}{\bf Theorem}[section]
\newtheorem{proposition}[theorem]{\bf Proposition}
\newtheorem{definition}[theorem]{\bf Definition}
\newtheorem{Theorem}{\bf Theorem}

\newtheorem*{claim}{\bf Claim}
\newtheorem{lemma}[theorem]{\bf Lemma}

\def\C{{\mathbb C}}

\def\R{{\mathbb R}}

\def\p{\mathbb{P}}

\def\dd{\mathrm{d}}

\title[Quantitative equidistribution]{Quantitative equidistribution of periodic points for rational maps}
\author{Thomas Gauthier}
\address{Laboratoire de Math\'ematiques d'Orsay, B\^atiment 307, Universit\'e Paris-Saclay, 91405 Orsay Cedex, France}
\email{thomas.gauthier@universite-paris-saclay.fr}
\author{Gabriel Vigny}
\address{LAMFA, Universit\'e de Picardie Jules Verne, 33 rue Saint-Leu, 80039 AMIENS Cedex 1, FRANCE}
\email{gabriel.vigny@u-picardie.fr}
\thanks{The first author is partially supported by the Institut Universitaire de France.}
\thanks{Both authors are partially supported by the French National Research Agency under the project DynAtrois, project ANR-24-CE40-1163.}

\begin{document}
\begin{abstract}
We show that periodic points of period $n$ of a complex rational map of degree $d$ equidistribute towards the equilibrium measure $\mu_f$ of the rational map with a rate of convergence of $(nd^{-n})^{1/2}$ for $\mathscr{C}^1$-observables. This is a consequence of a quantitative equidistribution of Galois invariant finite subsets of preperiodic points \`a la Favre and Rivera-Letelier.
Our proof relies on the H\"older regularity of the quasi-psh Green function of a rational map, an estimate of Baker concerning Hsia kernel, as well as on the product formula and its generalization by Moriwaki for finitely generated fields over $\mathbb{Q}$.
\end{abstract}

\maketitle


\section*{Introduction}
For a rational map $f:\p^1(\C)\to\p^1(\C)$ of degree $d\geq2$, Lyubich~\cite{lyubich:equi} defined a natural probability measure $\mu_f$ capturing the chaotic dynamics, called the \emph{equilibrium measure} of $f$. He proved that the measure $\mu_f$ is the unique measure of maximal entropy of $f$ and that it equidistributes periodic points. To make this statement more precise, for $n\geq1$, we denote by $\mathrm{Per}_n(f)$ the set of points for which $n$ is the smallest positive  integer with $f^{\circ n}(z)=z$. His result reads as
\[\frac{1}{d^n}\sum_{z\in \mathrm{Per}_n(f)}\delta_z\longrightarrow\mu_f,\]
as $n\to\infty$, where the convergence holds in the weak sense of measures on $\p^1(\C)$.

\medskip

 For  $0<\gamma\leq 1$, let $\mathscr{C}^\gamma(\mathbb{P}^1(\C),\R)$ denote the space of $\gamma$-Hölder observables, endowed with the usual norm $\|\cdot\|_{\mathscr{C}^\gamma}$. We write $\mathrm{Lip}(\mathbb{P}^1(\C),\R)$ for Lipschitz observables and $\mathrm{Lip}(\cdot)$ for the Lipschitz constant. 
Our aim here is to quantify the above convergence and give an exponential rate of convergence when tested against observables in $\mathscr{C}^\gamma(\mathbb{P}^1(\C),\R)$.
\begin{Theorem}\label{Main-theorem}
Let $f:\mathbb{P}^1(\C)\to\mathbb{P}^1(\C)$ be any complex degree $d\geq2$ rational map. Then for any $0<\gamma\leq 1$, there is a constant $C_{f,\gamma}>0$ depending only on $f$ and $\gamma$ such that
\[\left|\int_{\mathbb{P}^1(\C)}\varphi\, \dd\mu_f-\frac{1}{d^n}\sum_{z\in \mathrm{Per}_n(f)}\varphi(z)\right|\leq C_{f,\gamma}\cdot \left(\frac{n}{d^n}\right)^{\gamma/2}\cdot\|\varphi\|_{\mathscr{C}^\gamma},\]
for any $n\geq1$ and any $\varphi\in\mathscr{C}^\gamma(\mathbb{P}^1(\C),\R)$.
\end{Theorem}

The optimal rate of convergence would be of the order $d^{-n/2}$ when testing against $\mathscr{C}^1$ observables, so the rate we obtain is almost optimal. Moreover, it is the best convergence rate one can reach with our approach.
This result was already known for rational maps defined over a number field by Favre and Rivera-Letelier~\cite{FRL} (see also the work of Okuyama~\cite{Okuyama-Fekete} for a quantitative result and \cite{bilu,Baker-Rumely,thuil2005} for qualitative results of arithmetic nature). Recently, De Th\'elin, Dinh and Kaufmann~\cite{DeThelin-Dinh-Kaufmann} gave a non explicit exponential rate of convergence but for any endomorphism of a projective space of any dimension. Their proof relies on a pluripotential theoretic approach, whereas our proof relies on properties of configurations of points which behave similarly to Fekete configurations and on arithmetic properties of periodic points of a rational map. Both proofs are very different in nature.

\bigskip

Theorem~\ref{Main-theorem} actually follows from a more general result we describe now. For a rational map $f:\p^1\to\p^1$, let $F:\C^2\to\C^2$ be the only lift of $f$ such that $\mathrm{Res}(F)=1$. Then $f$ is defined over the field $K$ which is generated over $\mathbb{Q}$ by the coefficients of $F$. This is a finitely generated field over $\mathbb{Q}$ and preperiodic points of $f$ are defined in finite extensions of $K$. We thus can look at the action of the Galois group $\mathrm{Gal}(\bar{K}/K)$ on such points. 

Our second result is the following.

\begin{Theorem}\label{tm:Galois-quantitative}
Let $f$ be a degree $d\geq2$ rational map and let $K$ be the associated field as above. There is a constant $C>0$ depending only on $f$ such that for any $\varphi\in \mathrm{Lip}(\mathbb{P}^1(\C),\R)$ and any finite set $E\subset \p^1(\C)$ of preperiodic points of $f$ which is $\mathrm{Gal}(\bar{K}/K)$ invariant, then
\[\left|\frac{1}{\# S}\sum_{z\in E}\varphi(z)-\int_{\p^1(\C)}\varphi\, \dd\mu_f\right|\leq C\left(\frac{\log(\# E)}{\# E}\right)^{1/2}\cdot\mathrm{Lip}(\varphi).\]
\end{Theorem}

\noindent In the case where $f$ is defined over a number field, this is a special case of Theorem~7 from \cite{FRL} which involves the average of the canonical height function of $f$ along $E$. We probably can recover their Theorem~7, but we certainly can not obtain a result in terms of points of small height in the general case since our proof involves a choice of height function which depends a priori on $E$ (so only points of zero height are manageable). 

%

\bigskip

Let us now sketch the strategy of the proof of Theorem~\ref{tm:Galois-quantitative}.
The proof goes in two steps. First, for a probability measure $\mu$ with correctly normalized H\"older potential $g$ on $\p^1(\C)$ ($\mu=\omega+dd^c g$ where $\omega$ is the Fubini-Study measure), we define a notion \emph{quasi-$g$-Fekete configuration} $E\subset \p^1(\C)$ by assuming there is a constant $C>0$ independent of $n=\# E$ such that
\[\sum_{x\neq y\in F}\left(\log d(x,y)-g(x)-g(y)\right)\geq -C n\log(n),\]
where $ d(x,y)$ denotes the spherical distance between $x,y\in \p^1(\C)$. We prove an equidistribution estimate for the probability measures equidistributed on such finite configurations of points on $\p^1(\C)$, see Theorem~\ref{tm:speed-quasi-Fekete} below. A similar idea was introduced in the work of Pritsker~\cite{pritsker} in the case where $g$ is a potential of the equilibrium measure of a non-polar compact subset $K$ of $\C$ (see also the very recent paper \cite{levenberg2025equidistributionconjugatesalgebraicunits} where Fekete distributions are used to study equidistribution of Galois conjugate of algebraic integers). Our proof is strongly inspired by the complex contribution of Favre and Rivera-Letelier's approach~\cite{FRL} of the quantitative equidistribution of points of small heights.

\medskip

As the equilibrium measure of a rational map has H\"older-continuous potentials, to prove Theorem~\ref{Main-theorem}, we are left with proving that the set given by a finite $\mathrm{Gal}(\bar{K}/K)$-invariant set $E\subset \p^1(\C)$ is a quasi-$g_f$-Fekete configuration for an appropriate potential $g_f$ of the measure $\mu_f$. 
When $f$ is defined over a number field, we use the product formula to reduce the question to giving an upper bound of the corresponding quantity $\sum_{x\neq y\in E}\log d_v(x,y)-g_{f,v}(x)-g_{f,v}(y)$ for other absolute values $v$ of the number field over which $f$ is defined. The bound in question is given by a result of Baker~\cite{Baker-Green} which states that for any absolute value $v$ on $K$, we have $\sum_{x\neq y\in E}\log d_v(x,y)-g_{f,v}(x)-g_{f,v}(y)\leq C_v \# E\log(\# E)$ for an explicit constant $C_v$ depending on $f$ and on the absolute value.

 In the more challenging case where $f$ is not defined over a number field, we use that $f$ is defined over a finitely generated field over $\mathbb{Q}$ with positive transcendence degree. In this case, Moriwaki~\cite{Moriwaki} remarked that it can be identified with the field $\mathbb{Q}(\mathcal{B})$ of rational functions of a projective variety $\mathcal{B}$ which is flat over $\mathrm{Spec}(\mathbb{Z})$. Endowing $\mathcal{B}$ with a continuous hermitian line bundle $\bar{\mathscr{L}}$, Moriwaki proved a generalized product formula. We choose an appropriate hermitian metrization $\bar{\mathscr{L}}$ which depends on $E$ which allows to apply the same strategy as in the number field case. Using again Baker's estimate, we are able to obtain a constant independent of the choice of the metrization.\\

In higher dimension, it is known since the work of Briend and Duval~\cite{briendduval} that periodic points also equidistribute the equilibrium measure $\mu_f$ of an endomorphism $f$ of $\p^k(\C)$  as the period tends to infinity (see also Yuan~\cite{yuan} for a result of arithmetic nature). As recalled above, De Thélin, Dinh and Kaufmann~\cite{DeThelin-Dinh-Kaufmann} have recently shown that this convergence is exponential though the rate is unknown so it is natural to try and prove a comparable explicit rate of convergence. In this context, the only known explicit result is that of Yap~\cite{Yap} which is concerned with endomorphisms of $\p^2$ defined over a number field. His approach is quite different from our proof and relies on subtle quantitative estimates on Bergman kernels.
The work of Dinh, Ma and Nguyen~\cite{DMN}, where they estimate the speed of convergence for \emph{Fekete} configurations in dimension at least $2$ gives a direction for further explorations, since the equilibrium measure $\mu_f$ also has H\"older continuous potentials when $k\geq2$. However, there seems to be substantially more important difficulties in this context. In particular, Baker's estimate has only a partial generalization to higher dimensional varieties, see Looper~\cite{Looper-Green}. \\

The paper is organized as follows. In Section~\ref{sec:basic}, we give general preliminaries about the dynamics of rational maps on a field of characteristic zero (metrized or not). In Section~\ref{sec:qFekete}, we give a rate of convergence for quasi-Fekete configurations on $\p^1(\C)$. In Section~\ref{sec:Main}, we recall material on Moriwaki's generalized product formula on finitely generated fields over $\mathbb{Q}$ and we prove Theorems~\ref{Main-theorem} and~\ref{tm:Galois-quantitative}.

\subsection*{Acknowledgment}
The first author would like to thank Marc Abboud, Jean-Beno\^it Bost and Yugang Zhang for helpful discussions. 

\section{Preliminaries}\label{sec:basic}
\subsection{Resultant, lifts, dynatomic polynomials and discriminant}
We refer e.g.~to \cite{Silverman} for the material of this section.
 Let $K$ be a field of characteristic zero and $d\geq2$ be an integer. Let $F_1=\sum_{j}a_jX^jY^{d-j},F_2=\sum_{j}b_jX^jY^{d-j}\in K[X_1,X_2]$ be two homogeneous polynomials of degree $d$. The resultant of $F:=(F_1,F_2)$ is the determinant of the $2d$-square matrix given by
 \[\mathrm{Res}(F):=\det\left(\begin{array}{cccccccc}
 a_d & a_{d-1} & \cdots &\cdots & a_0 & 0&  \cdots & 0\\
 0 & a_d & a_{d-1}& \cdots & & \ddots & \ddots & \vdots \\
 \vdots & \ddots & \ddots &  \ddots & & & a_0&0\\
0 & \cdots &  0 & a_d  &a_{d-1} & \cdots& \cdots  &a_0\\
 b_d & \cdots &\cdots & b_1 & b_0& 0& \cdots  & 0\\
  0 & b_d& & & \ddots & \ddots  & \ddots & \vdots \\
\vdots & \ddots  & \ddots& &\cdots &  b_1 & b_0  & 0\\
0& \cdots &0&b_d&\cdots&\cdots &b_1&b_0 
 \end{array}
 \right)\in K.\]
 It satisfies $\mathrm{Res}(F)=0$ if and only if $F_1$ and $F_2$ have a common zero in $\mathbb{P}^1(\bar{K})$, if $c\in K$, then $\mathrm{Res}(cF)=c^{2d}\mathrm{Res}(F)$, and  $\mathrm{Res}(X^d,Y^d)=1$. 
\medskip

Let $f:\mathbb{P}^1_K\to\mathbb{P}^1_K$ be a degree $d$ rational map. A \emph{lift} of $f$ is a polynomial map $F:\mathbb{A}^2_K\to\mathbb{A}^2_K$ such that $\mathrm{Res}(F)\neq0$ and $f\circ \pi=\pi\circ F$ where $\pi:\mathbb{A}^2\setminus\{0\}\to\mathbb{P}^1$ is the natural projection. 

For any two homogeneous polynomial maps $F,H:\mathbb{A}^2_K\to\mathbb{A}^2_K$, we use the notation
 \[H\wedge F:=H_1F_2-H_2F_1\in K[X,Y],\]
 where $H=(H_1,H_2)$ and $F=(F_1,F_2)$.
 
 For any integer $n\geq1$, we let $d_n$ be the integer defined by
\[d_n:=\sum_{k|n}\mu\left(\frac{n}{k}\right)(d^k+1)\] 
where $\mu$ is the Möbius function and we define the $n$-th \emph{dynatomic polynomial} of $F$ as
\[\Psi_{F,n}(X,Y):=\prod_{\ell | n}\left(F^{\circ \ell}(X,Y)\wedge (X,Y)\right)^{\mu(n/\ell)}.\]
 The homogeneous polynomial $\Psi_{F,n}$ has degree $d_n$ and for $z=\pi(Z)$ with $Z\in \mathbb{A}^2(\bar{K})\setminus\{0\}$, we have $\Psi_{F,n}(Z)=0$ if and only if
\begin{itemize}
\item either $z$ is periodic for $f$ with exact period $n$, i.e. $f^{\circ n}(z)=z$ and $f^{\circ k}(z)\neq z$ for all $1\leq k<n$ (if such a $k$ exists),
\item or $z$ is a multiple periodic point for $f$ of period $m| n$ and the multiplier $(f^{\circ m})'(z)$ is a primitive $n/m$-root of unity in $\bar{K}$.
\end{itemize}
Moreover, by construction, if $a_d,\ldots,a_0, b_d,\ldots,b_0$ are the coefficients of $F$, then $\Psi_{F,n}$ has coefficients in $\mathbb{Z}[a_d,\ldots,a_0]$.
We let
\[\mathrm{Per}_n(f):=\{\pi(Z)\in \mathbb{P}^1(\bar{K})\, : \ \Psi_{F,n}(Z)=0\}\]
and we order the elements of $\mathrm{Per}_n(f)$ counting multiplicity as $\mathrm{Per}_n(f)=\{z_1,\ldots,z_{d_n}\}$. Finally, we can pick a finite collection of points $Z_1=(Z_{1,1},Z_{1,2}),\ldots,Z_{d_n}=(Z_{d_n,1},Z_{d_n,2})\in \mathbb{A}^2(\bar{K})\setminus\{0\}$ such that
\begin{align}
\Psi_{F,n}(X,Y)=\prod_{j=1}^{d_n}(X,Y)\wedge (Z_{j,1},Z_{j,2}),
\end{align}
as homogeneous polynomials in two variables $X,Y$. Define, in the field $\bar{K}$, 
\begin{equation}\label{eq:discriminant}
	\mathrm{disc}(\Psi_{F,n}):= \prod_{i\neq j}Z_i\wedge Z_j.
	\end{equation}
The first easy lemma here is the following (see~e.g.~\cite{Silverman}).
\begin{lemma}\label{lm:discriminant}
Assume $\mathrm{Res}(F)=1$. For any $n\geq1$, we have $\mathrm{disc}(\Psi_{F,n})\in K$ and $\mathrm{disc}(\Psi_{F,n})=0$ if and only if $f$ has a parabolic periodic point of period $m|n$ and  multiplier a $n/m$-primitive root of unity.
\end{lemma}



\subsection{Spherical distance, potentials and Hsia Kernel}\label{sec:Hsia}
Fix now an algebraically closed field $(K,|\cdot|)$ of characteristic zero which is complete with respect to a non-trivial absolute value. We refer to \cite{Baker-Rumely-book} for the material of this section.
The spherical distance on $\mathbb{P}^1(K)$ is defined as
\[d(x,y):=\frac{|X\wedge Y|}{\|X\|\cdot\|Y\|}, \quad x,y\in \mathbb{P}^1(K),\]
where $X,Y\in \mathbb{A}^2(K)\setminus\{0\}$ are such that $\pi(X)=x$ and $\pi(Y)=y$, where $\pi:\mathbb{A}^2\setminus\{0\}\to\mathbb{P}^1$ is the natural projection, and where
\[\|(X_1,X_2)\|=\left\{\begin{array}{ll}
\max\{|X_1|,|X_2|\} & \text{if} \ (K,|\cdot|) \ \text{is non-archimedean},\\
\sqrt{|X_1|^2+|X_2|}^2 & \text{if} \ (K,|\cdot|)\simeq(\C,|\cdot|) \ \text{is archimedean}.
\end{array}\right.\]
For any $x,y\in\mathbb{P}^1(K)$ we have $d(x,y)\leq 1$ and the induced function $d:\mathbb{P}^1(K)\times\mathbb{P}^1(K)\to\mathbb{R}_+$ extends to a continuous function $d:\mathbb{P}^{1,\mathrm{an}}_K\times \mathbb{P}^{1,\mathrm{an}}_K\to\mathbb{R}_+$, where $\mathbb{P}^{1,\mathrm{an}}_K$ is the Berkovich projective line over the field $K$. In the case where $K=\C$ is archimedean, the space $\mathbb{P}^{1,\mathrm{an}}_K$ is the Riemann sphere and $d$ is just the spherical distance.

\bigskip

When $(K,|\cdot|)\simeq(\C,|\cdot|)$, recall that the group spanned by rotations ($z\mapsto e^{i\theta} z$) and $PSU(2)$ (i.e. maps of the form $z\mapsto \frac{\cos \theta z -\sin \theta }{\sin \theta z +\cos \theta} $ for $\theta \in \R$) is a group of isometry for the spherical distance which acts bi-transitively on $\p^1$. In addition, the ball $B(0,r)$ centered at $0$ and of radius $r$ for the spherical distance correspond to the euclidean ball of center $0$ and radius $\frac{r}{\sqrt{1-r^2}}$.

 Let $z\in \p^1(\C)$ and $r>0$, pick an isometry $h$ that sends $0$ to $z$. In what follows, the normalized Lebesgue measure on the (spherical) circle $\partial B(z,r)$ of center $z$ and radius $r$ is defined as the direct image by $h$ of the normalized Lebesgue measure (i.e.\, of mass $1$) on $\partial B(0,r)$. Note that this is independent of the choice of $h$ because rotations preserve the Lebesgue measure on $\partial B(0,r)$.

\bigskip

Let now $\omega$ be either the Fubini-study form on $\mathbb{P}^1(\C)$ when $K$ is archimedean or the 
Dirac mass at the Gauss point $\zeta_{0,1}$ of $\p^{1,\mathrm{an}}_K$ when $K$ is non-archimedean. Recall that a probability measure $\mu$ on $\mathbb{P}^{1,\mathrm{an}}_K$ is said to have continuous potential if $\mu=\omega+\dd\dd^c g$ for some continuous function $g:\mathbb{P}^{1,\mathrm{an}}_K\to\mathbb{R}$. Note that such a function $g$ is unique up to additive constants. One can e.g. define $g$ as
\[g(x):=\int_{\mathbb{P}^{1,\mathrm{an}}_K}\log d(x,y)\dd\mu(y), \quad x\in\mathbb{P}^{1,\mathrm{an}}_K.\]
To such a function $g$, we can associate the Hsia kernel
\[\Phi_g(x,y):=\log d(x,y)-g(x)-g(y), \quad (x,y)\in \mathbb{P}^{1,\mathrm{an}}_K\times \mathbb{P}^{1,\mathrm{an}}_K.\]
The function $\Phi_g$ takes values in $\R\cup\{-\infty\}$ and satisfies $\dd\dd^c_y\Phi_g=\delta_x-\mu$ so that we have
\[\int_{\mathbb{P}^{1,\mathrm{an}}_K}\Phi_g(x,y)\,\dd\mu(y)=\iint_{\mathbb{P}^{1,\mathrm{an}}_K\times\mathbb{P}^{1,\mathrm{an}}_K}\Phi_g(w,y)\,\dd\mu(y)\dd\mu(w), \quad \text{for all} \ x\in \mathbb{P}^{1,\mathrm{an}}_K.\]

\subsection{The dynamical Hsia kernel and an estimate of Baker}\label{sec:dynmetrized}
Let $f:\mathbb{P}^1_K\to\mathbb{P}^1_K$ be a degree $d\geq2$ defined over $K$. Let $F:\mathbb{A}^2_K\to\mathbb{A}^2_K$ be a lift of $F$. The \emph{homogeneous Green function} $G_F$ of $F$ is
\[G_F:=\lim_{n\to\infty}d^{-n}\log\|F^{\circ n}\|,\]
where $\|X,Y\|=\sqrt{|X|^2+|Y|^2}$ if $(K,|\cdot|)\simeq(\C,|\cdot|)$ is archimedean and $\|X,Y\|=\max\{|X|,|Y|\}$ otherwise. The function $G_F-\log\|\cdot\|$ descends to $\mathbb{P}^1(K)$ as a continuous function which extends as a continuous function $g_F:\p^{1,\mathrm{an}}_K\to\R$. The measure
\[\mu_f:=\omega+\dd\dd^cg_F\]
is the \emph{equilibrium measure} of $f$ and it is independent of the choice of $F$. Moreover, the associated Hsia kernel, defined as $\Phi_{g_F}(x,y)=\log d(x,y)-g_F(x)-g_F(y)$ for $x,y\in \p^{1,\mathrm{an}}_K$, satisfies
\[\iint_{\p^{1,\mathrm{an}}_K\times \p^{1,\mathrm{an}}_K}\Phi_{g_F}(x,y)\dd\mu_f(x)\dd\mu_f(y)=-\frac{1}{d(d-1)}\log|\mathrm{Res}(F)|.\]
In particular, the following function is independent from the choice of lift
\begin{align}
\Phi_{g_f}(x,y):=\Phi_{g_F}(x,y)+\frac{1}{d(d-1)}\log|\mathrm{Res}(F)|.\label{eq-samelift}
\end{align}

For a rational map $f:\p^1\to\p^1$ defined over $K$, the \emph{good lift} of $f$ is the unique homogeneous polynomial lift $F:\mathbb{A}^2\to\mathbb{A}^2$ with $\mathrm{Res}(F)=1$.
The following follows from \cite{Baker-Green}:

\begin{lemma}\label{lm:BakerHsia}
There is a constant $C_d>0$ depending only on $d$ such that the following holds. For any complete algebraically closed field $(K,|\cdot|)$ of characteristic zero, any rational map $f$ of degree $d$ defined over $K$ and
 any collection $\{z_1,\ldots,z_N\}$ of $N$ distinct points of $\p^1(K)$, we have
\[\sum_{i\neq j}\Phi_{g_f}(z_i,z_j)\leq C_d\left(\epsilon_K+ \log^+\max\{|a_i|,|b_j|\} \right) N\log(N),\]
where $\epsilon_K=1$ is $K$ is archimedean and $\epsilon_K=0$ otherwise, and where $a_0,\ldots,a_d,b_0,\ldots,b_d\in K$ are the coefficients of the good lift $F$ of $f$.
\end{lemma}

\begin{proof} Let $F$ be the homogeneous lift of $f$ with $\mathrm{Res}(F)=1$. The quantity $D_\varphi$ in \cite{Baker-Green} is equal here to $- \sum_{i\neq j}\Phi_{g_f}(z_i,z_j)$ and for  $N=d^k$, \cite{Baker-Green}[(2.4) Corollary~2.3] implies
\[ \sum_{i\neq j}\Phi_{g_f}(z_i,z_j)\leq \epsilon_K N \log N+2 (d-1)k\log^+ R(F) N, \]
 where $R(F)$ is the smallest $R>0$ such that $\{G_F\leq 0\}\subset \bar{B}(0,R)$ as subsets of $K^2$ (indeed, $r(f)=0$ in our case, we can bound from above $\log R(F)$ by $\log^+ R(F)$,  and $\alpha=(d-1)k$ here since $\mathrm{Res}(F)=1$). Since $\sup_{z'_1, \dots, z'_N} \frac{1}{N(N-1)} \sum_{i\neq j}\Phi_{g_f}(z'_i,z'_j)$ is decreasing \cite{Baker-Green}[Lemma 2.13] and since for any integer $N\geq 1$, we can find  $N'$ of the form $N'=d^k$ with
\[ \frac{N-1}{d} \leq N'-1 \leq N-1,\]
we deduce	
\[ \frac{1}{N(N-1)} \sum_{i\neq j}\Phi_{g_f}(z_i,z_j)\leq \frac{1}{N'(N'-1)}  (\epsilon_K N' \log N'+2 (d-1)k\log^+ R(F) N') \]
which reads as
\[ \sum_{i\neq j}\Phi_{g_f}(z_i,z_j)\leq \frac{N(N-1)}{(N'-1)}  (\epsilon_K  \log N'+2 (d-1)k\log^+ R(F)). \]	
Hence
\[ \sum_{i\neq j}\Phi_{g_f}(z_i,z_j)\leq dN  (\epsilon_K  \log N+ 2(d-1)k\log^+ R(F)) \]	
So, using the bound $k\leq \log N$, we have	
\[ \sum_{i\neq j}\Phi_{g_f}(z_i,z_j)\leq d(\epsilon_K + 2(d-1)\log^+ R(F))N\log N. \]
We now apply the left-hand side of the second inequality of \cite[Lemma~2.2]{GOV2}: note that $g_F$ in  \cite[Lemma~2.2]{GOV2} is equal to $G_F-\log\|Z\|$ and the proof shows that the constant $A_1(d,K)$ of \cite[Lemma~2.2]{GOV2} is equal to $0$ if $K$ is non archimedean and depends only on $d$ when $K$ is archimedean. So, this means that for any $Z\in K^2\setminus\{(0,0)\}$, we have
\[-(2d-1)\log\max\{|a_i|,|b_j|\}-A_K\leq (d-1)G_F(Z)-(d-1)\log\|Z\|,\]
where $A_K>0$ depends only on $d$ when $K$ is archimedean and $A_K=0$ when $K$ is non-archimedean. In particular, if $G_F(Z)\leq 0$, we have 
\[\log\|Z\|\leq \frac{2d-1}{d-1}\log\max\{|a_i|,|b_j|\}+\frac{A_K}{d-1}\]
and
\[ \sum_{i\neq j}\Phi_{g_f}(z_i,z_j)\leq d(\epsilon_K +A_K+ 2(2d-1)\log\max\{|a_i|,|b_j|\})N\log N. \]
The conclusion follows.
\end{proof}

\begin{remark}\normalfont
 In particular, if $K$ is non-archimedean and $\max\{|a_i|,|b_j|\}\leq 1$, then 
\[\sum_{i\neq j}\Phi_{g_f}(z_i,z_j)\leq 0.\]
\end{remark}

\section{Distribution of quasi-Fekete configurations}\label{sec:qFekete}

Let $\mu$ be a probability measure on $\mathbb{P}^1(\mathbb{C})$ with continuous potential $g:\p^1(\C)\to\R$ i.e. $\mu=\omega+\dd\dd^cg$, where $\omega$ is the Fubini-Study form. Recall that
\[\Phi_g(x,y):=\log d(x,y)-g(x)-g(y)\in \R\cup\{-\infty\}, \quad (x,y)\in \p^1(\C)\times\p^1(\C)\]
is the Hsia kernel \cite{Baker-Rumely-book} associated with $g$. We normalize $g$ so that 
\[\int_{\mathbb{P}^1(\C)\times\mathbb{P}^1(\C)}\Phi_g(x,y)\,\dd\mu(x)\dd\mu(y)=0.\]
Such a $g$ will be called in the sequel the \emph{good potential} of $\mu$. By the above, this implies 
\[\forall x\in \p^1(\C), \  \int_{\mathbb{P}^1(\C)}\Phi_g(x,y)\dd\mu(y)=0.\]

\medskip

Pick $\alpha\in(0,1]$. We say that a function $g:\p^1(\C)\to\R$ is $\alpha$-H\"older continuous if
\[|g(z)-g(w)|\leq C d(z,w)^\alpha, \quad z,w\in \mathbb{P}^1(\mathbb{C})\]
for some constant $C>0$.
The $\mathscr{C}^{0,\alpha}$-semi-norm of $g$ is then defined as
\[\|g\|_{\mathscr{C}^{0,\alpha}}:=\sup_{x\neq y\in \p^1(\C)}\frac{|g(x)-g(y)|}{d(x,y)^\alpha}.\]
Denote by $\mathscr{C}^{\alpha}(\p^1(\C),\R)$ the vector space of all $\alpha$-H\"older continuous functions on $\p ^1(\C)$.
We say a probability measure $\mu$ on $\mathbb{P}^1(\mathbb{C})$ has \emph{$\alpha$-H\"older potential} if there is a function $g:\mathbb{P}^1(\mathbb{C})\to\mathbb{R}$ which is $\alpha$-H\"older continuous such that $\mu=\omega+\dd\dd^cg$.

\medskip

\begin{example}\normalfont
Let $f:\p^1(\C)\to\p^1(\C)$ is a rational map and let $F:\C^2\to\C^2$ be a lift of $f$. Then the function $g_F$ is $\alpha$-H\"older for some $\alpha\in(0,1]$ that depends only on $f$, see e.g.~\cite{Sibony}, whence $\mu_f$ has H\"older potentials.

\medskip

Moreover, if we choose $F$ so that $\mathrm{Res}(F)=1$, the function $g_F$ is the good potential of the measure $\mu_f$ by \S~\ref{sec:dynmetrized}. We denote $g_f:=g_F$ this good potential.
\end{example}

\begin{definition}
Let $\mu$ be a probability measure on $\mathbb{P}^1(\C)$ with continuous potential and let $g$ be its good potential. Pick a constant $C>0$. A finite set $F\subset\mathbb{P}^1(\C)$ of cardinality $n\geq2$ is a \emph{quasi-$g$-Fekete configuration} with constant $C$ if 
\[\sum_{x\neq y\in F}\Phi_g(x,y)\geq -C n \log(n).\]
\end{definition}

\begin{remark}\normalfont
Our normalization of the good potential $g$ of the measure $\mu$ means that the energy of the potential is $0$. Recall that a $g$-Fekete configuration of order $n$ is a collection $F\subset \p^1(\C)$ of $n$ distinct points attaining the maximum of the quantity $\sum_{x\neq y\in F}\Phi_g(x,y)$ among finite subsets $F\subset \p^1(\C)$ with $\# F=n$. A simple application of Fubini's Theorem implies that if $F$ is a $g$-Fekete configuration of order $n$, then
\[\sum_{x\neq y\in F}\Phi_g(x,y)\geq0.\]
In particular, a $g$-Fekete configuration of order $n$ is quasi-$g$-Fekete for any constant $C>0$.
\end{remark}

 We denote by $\|\cdot \|_\alpha$ the semi-norm on $W^{1,2}(\mathbb{P}^1(\C),\mathbb{R}) \cap \mathscr{C}^{\alpha} (\mathbb{P}^1(\C), \R)$ defined by 
\[\|\varphi\|_\alpha=\|\varphi\|_{\mathscr{C}^{0,\alpha}}+\|\nabla\varphi\|_{\mathrm{L}^2}.\]
The aim of this section is to give a relatively simple proof of the following.
\begin{theorem}\label{tm:speed-quasi-Fekete}
Let $\mu$ be a probability measure on $\mathbb{P}^1(\C)$ with $\alpha$-H\"older good potential $g$ and pick $C'>0$. Then for any $\gamma\in [\min(1/4,2\alpha),1]$, any $\varphi \in W^{1,2}(\mathbb{P}^1(\C),\mathbb{R}) \cap \mathscr{C}^{\gamma} (\mathbb{P}^1(\C), \R)$ and any quasi-$g$-Fekete configuration $F\subset\mathbb{P}^1(\C)$ of cardinality $n\geq2$ with constant $C'$, 
 \begin{align*}
  & \left| \int_{\mathbb{P}^1(\C)} \varphi \, \mathrm{d} \mu -  \frac{1}{n} \sum_{z\in F} \varphi (z) \right| \leq C_1\left(\frac{\log(n)}{n}\right)^{1/2}\|\varphi\|_\gamma,
 \end{align*}
where we can take $C_1=\sqrt{2\|g\|_{\mathscr{C}^{0,\alpha}}+1+C'+\max(2,1/\alpha)}$.
\end{theorem}


\begin{remark}\normalfont
The proof also gives a rate of convergence for less regular functions. Indeed, if $\varphi\in W^{1,2}(\p^1(\C),\R)\cap \mathscr{C}^{\gamma}(\p^1(\C),\R)$ for $\gamma\in(0,\min(1/4,2\alpha)$, we prove
 \begin{align*}
  & \left| \int_{\mathbb{P}^1(\C)} \varphi \, \mathrm{d} \mu -  \frac{1}{n} \sum_{z\in F} \varphi (z) \right| \leq C_1\left(\frac{\log(n)}{n}\right)^{1/2}\|\nabla\varphi\|_{\mathrm{L}^2}+\frac{1}{n^{\gamma\max(2,1/\alpha)}}\cdot\|\varphi\|_{\mathscr{C}^{0,\gamma}}.
 \end{align*}
\end{remark}

\subsection{A result \`a la Favre-Rivera-Letelier}
Theorem~\ref{tm:speed-quasi-Fekete} is a direct consequence of the following proposition inspired by Favre and Rivera-Letelier~\cite{FRL}:

\begin{proposition}\label{tm:speed-FRL}
Pick $\alpha \in (0,1]$ and let $\mu$ be a probability measure with $\alpha$-H\"older continuous potential $g$. For all $\varphi \in \mathscr{C}^0(\p^1(\C),\R)\cap W^{1,2}(\mathbb{P}^1(\C),\mathbb{R})$ with modulus of continuity $\eta_\varphi$, any finite set $F\subset\mathbb{P}^1(\C)$ of cardinality $n\geq2$, we have
 \begin{align*}
  & \left| \int_{\mathbb{P}^1(\C)} \varphi \, \mathrm{d} \mu -  \frac{1}{n} \sum_{z\in F} \varphi (z) \right| \leq  \eta_\varphi\left(\frac{1}{n^{\max(2,1/\alpha)}}\right)+A^{1/2}\cdot \|\nabla\varphi\|_{\mathrm{L}^2} ,
 \end{align*}
 where we can compute $A=A(g,C,\alpha,F)$ as
 \[A=-\frac{1}{n(n-1)}\sum_{x\neq y\in F}\Phi_g(x,y)+\left(2\|g\|_{\mathscr{C}^{0,\alpha}}+2+\max\left(2,\frac{1}{\alpha}\right)\right)\frac{\log(n)}{n}.\]
\end{proposition}

Theorem~\ref{tm:speed-quasi-Fekete} follows directly from Proposition~\ref{tm:speed-FRL} since for a quasi-$g$-Fekete configuration with constant $C'$, we have $-\sum_{x\neq y\in F}\Phi_g(x,y)\leq C' n \log(n)$ so the constant $A$ in Proposition~\ref{tm:speed-FRL} is $O\left(\frac{\log(n)}{n}\right )$, with explicit constants.

\bigskip

Given any two probability measures $\mu$ and $\nu$ on $\mathbb{P}^1$, the mutual energy of $\mu$ and $\nu$ is
\[(\mu,\nu):=-\int_{\mathbb{P}^1(\C)\times\mathbb{P}^1(\C)}\log d(x,y) \, \dd\mu(x)\dd\nu(y)\in \R_+\cup\{+\infty\}.\]
The mutual energy of $\mu$ and $\nu$ is symmetric and it is finite when $\mu$ and $\nu$ have continuous potential, or when $\mu$ has continuous potential and $\nu$ is a finite sum of atoms. We extend this energy by bilinearity to signed measures of finite mass. The idea to use the mutual energy to quantify convergence of measures has been introduced by Favre and Rivera-Letelier~\cite{FRL} (see also \cite{oberly} for different arithmetic applications).

The next lemma follows easily from the work of Favre and Rivera-Letelier~\cite{FRL}:

\begin{lemma}[Favre-Rivera-Letelier]\label{lm:FRL}
Let $\rho = \dd\dd^c u$ be a signed measure of mass $0$, i.e. such that $\rho(\mathbb{P}^1(\mathbb{C}))=0$. Assume $u \in \mathscr{C}^0 (\mathbb{P}^1(\C), \R)$. Then $u \in W^{1,2}(\mathbb{P}^1(\mathbb{C}))$ and we have
 \begin{equation*}
(\rho,\rho)  = \int_{\mathbb{P}^1(\C)} \dd u \wedge \dd^c u
   \geq 0.
  \end{equation*}
  In addition, $(\rho,\rho) = 0$ if and only if $\rho = 0$.
 \end{lemma}
\begin{proof}
The result is proved in \cite[Proposition 2.6]{FRL} for another choice of mutual energy pairing. So we just need to show that, in this case, they coincide.

Favre and Rivera-Letelier's definition of the mutual energy is
\[(\mu,\nu)_{\mathrm{FRL}}:=-\int_{\C\times\C}\log |x-y| \, \dd\mu(x)\dd\nu(y).\]
Note that $\int_{\mathbb{P}^1(\C)}\log |x-y| \dd\nu(y)$ is the logarithmic potential of $\nu$. In particular, when $\nu=\omega$, for any $x\in \C$ we have:
 \[ \int_{\C}\log |x-y| \dd\omega(y) = \frac{1}{2}\log (1+|x|^2). \]
For any probability measures $\mu$ and $\nu$ with continuous potentials, we claim 
 \[(\mu,\nu)=(\mu,\nu)_{\mathrm{FRL}}-(\omega,\nu)_{\mathrm{FRL}}-(\omega,\mu)_{\mathrm{FRL}}.\]
Indeed, by Fubini and $\log d(x,y)= \log |x-y| -  \frac{1}{2}\log (1+|x|^2)- \frac{1}{2}\log (1+|y|^2)$:
\begin{align*}
	(\mu,\nu)&= (\mu,\nu)_{\mathrm{FRL}} + \frac{1}{2}\int_{\C^2} \log (1+|x|^2)\dd\mu(x)\dd\nu(y) +\frac{1}{2}\int_{\C^2} \log (1+|y|^2) \dd\mu(x)\dd\nu(y) \\
	        &= (\mu,\nu)_{\mathrm{FRL}}+  \frac{1}{2}\int_{\C} \log (1+|x|^2) \dd\mu(x) +\frac{1}{2}\int_{\C} \log (1+|y|^2)\dd\nu(y)\\
	        &= (\mu,\nu)_{\mathrm{FRL}}+  \int_{\C^2} \log |x-y| \dd\mu(x)  \dd\omega(y)+ \int_{\C^2} \log |x-y| \dd\omega(x)  \dd\nu(y)\\
	        & =(\mu,\nu)_{\mathrm{FRL}}-(\omega,\nu)_{\mathrm{FRL}}-(\omega,\mu)_{\mathrm{FRL}}.
\end{align*}
By bilinearity of the mutual energy, we have
 \[(\mu-\nu,\mu-\nu)=(\mu-\nu,\mu-\nu)_{\mathrm{FRL}}=\int_{\mathbb{P}^1(\C)}\dd u\wedge \dd^cu,\]
 where $\mu-\nu=\dd\dd^cu$. Thus Lemma~\ref{lm:FRL} holds.
 \end{proof}


\subsection{The proof of Proposition~\ref{tm:speed-FRL}}

Write $F=\{z_1,\ldots,z_n\}$. Write $\nu_n:=\frac1{n}\sum_i\delta_{z_i}$ and fix $\varepsilon > 0$. Denote by $\mu_{i, \varepsilon}$ the uniform measure on $\partial B(z_i, \varepsilon)$, where $B(z_i,\varepsilon)$ is the ball for the spherical distance centered at $z_i$ and of radius $\varepsilon$, and by $\nu_{n, \varepsilon}$ the average measure $\frac1{n}\sum_i\mu_{i, \varepsilon}$ (see above for the construction of such measures). Then
\begin{align*}
  \left| \int_{\mathbb{P}^1(\C)} \varphi\, \dd \mu - \int_{\mathbb{P}^1(\C)} \varphi \,\dd \nu_n \right| 
  \leq & \left| \int_{\mathbb{P}^1(\C)}\varphi \,\dd \mu - \int_{\mathbb{P}^1(\C)} \varphi \,\dd \nu_{n, \varepsilon} \right|\\
&  \quad + \left| \int_{\mathbb{P}^1(\C)}\varphi \,\dd \nu_{n, \varepsilon} - \int_{\mathbb{P}^1(\C)}\varphi \,\dd \nu_n \right|.
 \end{align*}
  The latter term of the right hand-side is at most $\eta_\varphi(\varepsilon)$. Pick $g_{n,\varepsilon}\in \mathscr{C}^0(\mathbb{P}^1(\C),\R)$ such that $\nu_{n,\varepsilon}-\omega=\dd\dd^cg_{n,\varepsilon}$. Then, by Stokes formula
  \begin{equation*}
  \int_{\mathbb{P}^1(\C)} \varphi \, \dd \mu - \int_{\mathbb{P}^1(\C)} \varphi \, \dd \nu_{n, \varepsilon} 
  = \int_{\mathbb{P}^1(\C)} \varphi \, \dd \dd^c (g-g_{n,\varepsilon}) 
  = - \int_{\mathbb{P}^1(\C)} \dd \varphi \wedge \dd^c (g-g_{n,\varepsilon}).
 \end{equation*}
 Lemma~\ref{lm:FRL} and Cauchy-Schwartz inequality thus imply
\begin{align*}
  \left| \int_{\mathbb{P}^1(\C)} \varphi \,\dd \mu - \int_{\mathbb{P}^1(\C)} \varphi \,\dd \nu_{n, \varepsilon} \right|^2 
  & = \left| \int_{\mathbb{P}^1(\C)} \dd \varphi \wedge \dd^c (g-g_{n, \varepsilon}) \right|^2 \\
  & \leq \left(\int_{\mathbb{P}^1(\C)} \dd \varphi \wedge \dd^c \varphi\right)\times \left(\int_{\mathbb{P}^1(\mathbb{C})}\dd(g-g_{n, \varepsilon} )\wedge \dd^c(g-g_{n, \varepsilon})\right).
 \end{align*}
We are left with estimating the last integral which is $(\mu-\nu_{n, \varepsilon},\mu-\nu_{n, \varepsilon})$:
 \begin{align*}
 (\mu-\nu_{n, \varepsilon},\mu-\nu_{n, \varepsilon})
  & = (\mu,\mu)+(\nu_{n,\varepsilon},\nu_{n,\varepsilon})-2 (\mu,\nu_{n, \varepsilon}).
 \end{align*}
 We now use the following.
 \begin{claim}
$\displaystyle (\nu_{n, \varepsilon},\nu_{n, \varepsilon}) 
  \leq  -\frac{2}{n(n-1)}\sum_{1 \leq i < j \leq n} \log d(z_i,z_j) -\frac{1}{n}\log(\varepsilon)+2\sqrt{\varepsilon}$.
 \end{claim}
Taking the claim for granted, we finish the proof of Proposition~\ref{tm:speed-FRL}.
 Since for any $y\in \mathbb{P}^1(\C)$ we have $\int_{\mathbb{P}^1(\C)} \log d(x,y)d\mu(x)= g(y)+c$ where $c$ is a constant and since $g$ is a $\alpha$-H\"older continuous function, we have by Fubini 
 \begin{align*}
-(\mu,\nu_{n, \varepsilon}) &= -(\mu,\nu_{n, \varepsilon}-\nu_n) -(\mu,\nu_n)\\
& \leq \|g\|_{\mathscr{C}^{0,\alpha}}\varepsilon^\alpha-(\mu,\nu_n).
 \end{align*}
Together with the Claim, this summarizes as
  \begin{align}
 \label{first-ineg}(\mu-\nu_{n, \varepsilon},\mu-\nu_{n, \varepsilon})
  \leq (\mu,\mu)& +2\sqrt{\varepsilon}+2\|g\|_{\mathscr{C}^{0,\alpha}}\varepsilon^\alpha-\frac{2}{n}\log(\varepsilon)-2(\mu,\nu_n)\\  
  & -\frac{2}{n(n-1)}\sum_{1\leq i<j\leq n}\log d(z_i,z_j).\nonumber
 \end{align}
We now use that $g$ is the good potential of $\mu$ to see that
\[(\mu,\mu)=-\int_{\mathbb{P}^1(\C)\times\mathbb{P}^1(\C)}\log d(x,y)\, \dd\mu(x)\dd\mu(y)=-2\int_{\mathbb{P}^1(\C)}g\, \dd\mu.\]
Similarly, we can compute
\begin{align*}
	(\mu,\nu_n)&=-\int_{\mathbb{P}^1(\C)\times\mathbb{P}^1(\C)}\log d(x,y)\, \dd\nu_n(x)\dd\mu(y)\\
	           &=\int_{\mathbb{P}^1(\C)\times\mathbb{P}^1(\C)} - \Phi_g\, \dd\nu_n(x)\dd\mu(y) - \int_{\mathbb{P}^1(\C)\times\mathbb{P}^1(\C)} (g(x)+g(y)) \dd\nu_n(x)\dd\mu(y)\\
	        	           &=-\frac{1}{n}\sum_{i=1}^ng(z_i)-\int_{\mathbb{P}^1(\C)}g\, \dd\mu.
\end{align*}
Combining those two formulas with \eqref{first-ineg} gives
  \begin{align*}
 (\mu-\nu_{n, \varepsilon},\mu-\nu_{n, \varepsilon})
&  \leq  2\|g\|_{\mathscr{C}^{0,\alpha}}\varepsilon^\alpha+2\sqrt{\varepsilon}\\
  & \quad-\frac{2}{n}\log(\varepsilon)+\frac{2}{n}\sum_{j=1}^ng(z_j)-\frac{2}{n(n-1)}\sum_{1\leq i<j\leq n}\log d(z_i,z_j).
 \end{align*}
The conclusion follows picking $\varepsilon:=n^{-1/{\min(\alpha, 1/2)}}$ and using the definition of $\Phi_g$.

~

We are left with proving the claim:
\begin{proof}[Proof of the Claim]
 By definition and bilinearity of the mutual energy, we have
 \begin{align*}
 (\nu_{n, \varepsilon},\nu_{n, \varepsilon}) 
  &  = \frac{1}{n^2} \sum_{i=1}^n (\mu_{i, \varepsilon} ,\mu_{i, \varepsilon})  + \frac{2}{n^2} \sum_{1 \leq i < j \leq n} (\mu_{i, \varepsilon},\mu_{j, \varepsilon}).
 \end{align*}
 For a given $j$, we look at
 \[(\mu_{j, \varepsilon},\mu_{j, \varepsilon})= - \int_{\mathbb{P}^1(\C)\times\mathbb{P}^1(\C)}  \log d(x,y) \dd \mu_{i, \varepsilon}(x) \dd \mu_{i, \varepsilon}(y). \]
 Composing by an isometry of the spherical distance, we can assume that $z_i=0$. 
 We can assume that $\varepsilon$ is small enough so that $\log d(x,y)\geq 2\log |x-y|$ for $x,y \in B(0,\varepsilon)$. In particular
  \begin{align*}(\mu_{j, \varepsilon},\mu_{j, \varepsilon})&= - \int_{\mathbb{P}^1(\C)\times\mathbb{P}^1(\C)}  \log d(x,y) \dd \lambda_{\frac{\varepsilon}{\sqrt{1-\varepsilon^2}}}(x)\dd \lambda_{\frac{\varepsilon}{\sqrt{1-\varepsilon^2}}}(y)\\
  	&\leq  - 2\int_{\mathbb{P}^1(\C)\times\mathbb{P}^1(\C)}  \log |x-y| \dd \lambda_{\frac{\varepsilon}{\sqrt{1-\varepsilon^2}}}(x)\dd \lambda_{\frac{\varepsilon}{\sqrt{1-\varepsilon^2}}}(y).
  \end{align*}
  Now, $G_\varepsilon(y):=\int_{\mathbb{P}^1(\C)}  \log |x-y| \dd \lambda_{\frac{\varepsilon}{\sqrt{1-\varepsilon^2}}}(x)$ is the logarithmic potential of $\lambda_{\frac{\varepsilon}{\sqrt{1-\varepsilon^2}}}$ which is equal to $\log\left(\frac{\varepsilon}{\sqrt{1-\varepsilon^2}}\right)$ on the spherical ball $B(0,\varepsilon)$. In particular
    \begin{align*}(\mu_{j, \varepsilon},\mu_{j, \varepsilon}) 	&\leq  - 2\int_{\mathbb{P}^1(\C)}  \log\left(\frac{\varepsilon}{\sqrt{1-\varepsilon^2}}\right)\dd \lambda_{\frac{\varepsilon}{\sqrt{1-\varepsilon^2}}}(y) \leq -2\log \varepsilon.
  \end{align*}
Pick now a term with $i\neq j$. Assume first that $ d(z_i,z_j)\geq \sqrt{\varepsilon}$ so, for $x,y \in \partial B(z_i,\varepsilon) \times \partial B(z_j,\varepsilon)$ we have
$|d(x,y)-d(z_i,z_j)|< d(z_i,z_j)$ (provided $\varepsilon$ is small enough) and $|d(x,y)-d(z_i,z_j)|\leq 2 \varepsilon$ from the triangle inequality. From $\log (1+t) \leq t$ we deduce
\begin{align*}
	\log |d(x,y)|&= \log |d(z_i,z_j)+d(x,y)-d(z_i,z_j)|\\
	&\leq  \log |d(z_i,z_j)| + \log (1 + \left|\frac{ d(x,y)-d(z_i,z_j)}{d(z_i,z_j)}\right|) \\
	&\leq \log |d(z_i,z_j)| + \left|\frac{ d(x,y)-d(z_i,z_j)}{d(z_i,z_j)}\right|
\end{align*}
so
\begin{equation}\label{eq_cas_loin}\log |d(x,y)|\leq \log |d(z_i,z_j)|+ 2\sqrt{\varepsilon}. 
\end{equation}
Suppose now that $ d(z_i,z_j)\leq \sqrt{\varepsilon}$. 
Again, composing by an isometry of the spherical distance, we can assume that $z_i=0$.
The spherical balls $B(0, \varepsilon)$ and $B(z_j,\varepsilon)$ are contained in the (euclidean) disk of center $0$ and radius $2\sqrt{\varepsilon}$. So for $x,y \in \partial B(0,\varepsilon) \times \partial B(z_j,\varepsilon)$, by the inequality $\log(1 +t) \leq t$.
 \begin{align*} -\log |d(x,y)|&= -\log |x-y|+\frac{1}{2}\log(1+|x|^2)+ \frac{1}{2}\log(1+|y|^2) \\
 	                          &\leq -\log |x-y|+ 4 \varepsilon 
 \end{align*}
 In particular
 \begin{align*}
 (\mu_{i,\varepsilon},\mu_{j,\varepsilon})\leq -\int_{\mathbb{P}^1(\C)\times\mathbb{P}^1(\C)}  \log |x-y| \dd \mu_{i, \varepsilon}(x) \dd \mu_{j, \varepsilon}(y)+4\varepsilon. 
 \end{align*}
 where  $\dd \mu_{i, \varepsilon}$ is the normalized Lebesgue measure $\lambda_{\frac{\varepsilon}{\sqrt{1-\varepsilon^2}}}$ on the circle of center $0$ and radius $\frac{\varepsilon}{\sqrt{1-\varepsilon^2}}$.
As before $G_\varepsilon(y):=\int_{\mathbb{P}^1(\C)}  \log |x-y| \dd \lambda_{\frac{\varepsilon}{\sqrt{1-\varepsilon^2}}}(x)$ is the logarithmic potential of $\lambda_{\frac{\varepsilon}{\sqrt{1-\varepsilon^2}}}$ which satisfies  $G_\varepsilon(y)\geq \log|y|$. By the submean inequality, we deduce
 \begin{align*}
	(\mu_{i,\varepsilon},\mu_{j,\varepsilon})\leq - \int_{\mathbb{P}^1(\C)} G_\varepsilon(y) \dd \mu_{j, \varepsilon}(y)+4\varepsilon \leq -G(z_j)+4\varepsilon\leq -\log|z_j|+4\varepsilon. 
\end{align*}
Since $d(0,z_j)\leq |z_j|$, we infer $(\mu_{i,\varepsilon},\mu_{j,\varepsilon})\leq  -\log d(z_i,z_j)+4\varepsilon$. Combining this inequality and \eqref{eq_cas_loin}, we deduce that for every $i\neq j$, we have (again taking $\varepsilon$ small enough)
 \[  (\mu_{i,\varepsilon},\mu_{j,\varepsilon})\leq -\log d(z_i,z_j)+2\sqrt{\varepsilon}.\]
We have $(\mu_{i,\varepsilon},\mu_{j,\varepsilon})\geq0$, so summing the above
\begin{align*}
 \frac{2}{n^2} \sum_{1 \leq i < j \leq n} (\mu_{i, \varepsilon},\mu_{j, \varepsilon}) & \leq \frac{2}{n(n-1)}\sum_{1 \leq i < j \leq n} (\mu_{i, \varepsilon},\mu_{j, \varepsilon})\\
 &\leq -\frac{2}{n(n-1)}\sum_{1 \leq i < j \leq n} \log d(z_i,z_j) +2\sqrt{\varepsilon}.
 \end{align*}
 This concludes the proof.
 \end{proof}

\section{Galois invariant finite sets of preperiodic points are quasi-Fekete}\label{sec:Main}
 
Let $f:\p^1\to\p^1$ be a complex rational map of degree $d\geq2$ and let $F$ be its good lift, i.e. the lift with $\mathrm{Res}(F)$ and let $a_0,\ldots,a_d,b_0,\ldots,b_d$ be its coefficients. Let $K:=\mathbb{Q}(a_0,\ldots,a_d,b_0,\ldots,b_d)$. The field $K$ is finitely generated over $\mathbb{Q}$.

The main result of this section is the following.
\begin{theorem}\label{tm:Gal-Fekete}
There is a constant $C>0$ depending only on $f$ such that any finite set $E\subset \p^1(\C)$ of pre-periodic points of $f$ which is invariant by $\mathrm{Gal}(\bar{K}/K)$ is quasi-$g_f$-Fekete with constant $C$. More precisely, we have
\[-C\# E\log(\# E)\leq \sum_{z\neq w\in E}\Phi_{g_f}(z,w)\leq C\# E\log(\# E).\]
\end{theorem}

Note that the right-hand side inequality is just an immediate application of Lemma~\ref{lm:BakerHsia}. We thus focus on proving the other inequality in this section.

 \subsection{Generalized product formula}
 \subsubsection{The case of number fields: the product formula}
 Let $\mathbb{K}$ be a number field and let $M_\mathbb{K}$ be the set of places of $\mathbb{K}$, i.e. the set of equivalence classes of absolute values. For each $v\in M_\mathbb{K}$, we let $|\cdot|_v$ be an absolute value, normalized so that we have the product formula
\begin{align*}
\sum_{v\in M_\mathbb{K}}N_v\log|x|_v=0, \quad x\in \mathbb{K}^\times,
\end{align*}
where $N_v=[\mathbb{K}_v:\mathbb{Q}_p]$ with $\mathbb{K}_v$ the completion of $(\mathbb{K},|\cdot|_v)$, and $p$ is the residual characteristic of $(\mathbb{K},|\cdot|_v)$, when $v$ is non-archimedean and where $N_v=[\mathbb{C}:\mathbb{R}]=2$ or $N_v=[\mathbb{R}:\mathbb{R}]=1$ when $v$ is archimedean. We let $M_\mathbb{K}^\infty$ denote the set of archimedean places.

\medskip

When $\mathbb{L}$ is a finite extension of $\mathbb{K}$, any absolute value $|\cdot|_v$ on $\mathbb{K}$ extends to $\mathbb{L}$, but the extension is a priori not unique. If $|\cdot|_w$ is an absolute value on $\mathbb{L}$ which restricts to $\mathbb{K}$ as $|\cdot|_v$, we use the notation $w|v$. Fix $v\in M_\mathbb{K}$, one has
\[\sum_{w\in M_\mathbb{L}, \ w|v}[\mathbb{L}_w:\mathbb{K}_v]=[\mathbb{L}:\mathbb{K}].\]
The product formula on $\mathbb{L}$ can thus be written
\begin{align}
	\sum_{v\in M_\mathbb{K}}	 N_v \sum_{w\in M_\mathbb{L}, \  w|v} [\mathbb{L}_w:\mathbb{K}_v] \log|x|_w=0, \quad x\in \mathbb{L}^\times.\label{productformulanbfield2}
\end{align}
We refer to \cite{Lang-algebra} for more details.

\subsubsection{Generalized product formula for finitely generated fields}\label{Sec:finitelygenerated}
Let now $K$ be a finitely generated field over $\mathbb{Q}$, i.e. there are $T_1,\ldots,T_N$ such that $K=\mathbb{Q}(T_1,\ldots,T_N)$. Assume $K$ has transcendence degree $\mathrm{trdeg}(K/\mathbb{Q})\geq1$ over $\mathbb{Q}$. Then there is a normal projective variety $\mathcal{B}$, flat over $\mathrm{Spec}(\mathbb{Z})$ with function field $K$ (see,~e.g.,~\cite{Moriwaki}). Let $\mathcal{B}^{(1)}$ be the set of all integral closed sub-schemes $\Gamma\subset\mathcal{B}$ of codimension $1$ in $\mathcal{B}$ (either vertical or horizontal). Let also
	\[\mathcal{B}^{\mathrm{gen}}(\mathbb{C}):= \mathcal{B}(\C)\setminus \bigcup_{\Gamma\in \mathcal{B}^{(1)}}\mathrm{supp}(\Gamma).\] 
For every $\Gamma\in\mathcal{B}^{(1)}$, we can define an absolute value
\[|\cdot|_\Gamma:=\exp(-\mathrm{ord}_\Gamma) :K\longrightarrow\mathbb{R}_+,\]
where $\mathrm{ord}_\Gamma :K^\times\longrightarrow\mathbb{Z}$ is the valuation defined by the order of vanishing
along $\Gamma$. Every non-archimedean absolute value on $K$ is obtained this way. Moreover, for every prime number $p$, the set 
\begin{center}
$\mathcal{B}^{(1)}_p:=\{\Gamma\subset\mathcal{B}^{(1)}$, the residue characteristic of $(K_\Gamma,|\cdot|_\Gamma)$ is $p\}$
\end{center}
is a finite set. Finally, any point $b\in \mathcal{B}^{\mathrm{gen}}(\mathbb{C})$ corresponds to an absolute value $|\cdot|_b$ on $K$ defined by 
\[|u|_b:=|u(b)|_\mathbb{C}, \quad u\in K.\]
Those absolute values are archimedean and all archimedean absolute values on $K$ are obtained this way. Whence, the set of places $M_K$ on $K$ is  
\[M_K=\mathcal{B}^{\mathrm{gen}}(\mathbb{C}) \sqcup \bigsqcup_{\Gamma\in \mathcal{B}^{(1)}}\{|\cdot|_\Gamma\}.\]

We let $\bar{\mathscr{L}}$ be a Hermitian line bundle on $\mathcal{B}$, i.e. a big and nef line bundle $\mathscr{L}$ on $\mathcal{B}$ endowed with a semi-positive and continuous metrization $\|\cdot\|$ on $\mathscr{L}(\mathbb{C})$ such that $\bar{\mathscr{L}}$ is arithmetically nef in the sense of Zhang~\cite{Zhang-positivity}. One can then define a positive measure on $\mathcal{B}(\mathbb{C})$ of mass $\deg_\mathscr{L}(\mathcal{B})$ as
\[\mu_{\bar{\mathscr{L}}}:=c_1(\mathscr{L},\|\cdot\|)^{\wedge \dim \mathcal{B}}.\]
Remark that, by Bedford-Taylor, the measure $\mu_{\bar{\mathscr{L}}}$ does not give mass to closed subvarieties of $\mathcal{B}(\mathbb{C})$.
Fix some prime number $p$. For $\Gamma\in \mathcal{B}^{(1)}$, we define $a_\Gamma$ as the arithmetic intersection number
\[a_\Gamma:=\left(\left(\bar{\mathscr{L}}\right)^{\dim \mathcal{B}}|\Gamma\right)\geq0.\]
Since $\pi:\mathcal{B}\to\mathrm{Spec}(\mathbb{Z})$ is flat, and since the set $\mathcal{B}_p^{(1)}$ coincides with the set of irreducible components of $\pi^{-1}\{p\}$, we necessarily have
\begin{align}
\sum_{\Gamma\in \mathcal{B}^{(1)}_p}a_\Gamma=\deg_{\mathscr{L}}(\mathcal{B}).\label{eq:finiteplace}
\end{align}
In particular, for any choice of hermitian metric on $\mathscr{L}$ for which $\bar{\mathscr{L}}$ is nef, we have
\begin{align*}
0\leq a_\Gamma\leq \deg_{\mathscr{L}}(\mathcal{B}), \quad \text{for any} \ \Gamma\in \mathcal{B}^{(1)}.
\end{align*}
Following~\cite[\S~3.2]{Moriwaki}, we have the following generalization of the product formula:
\begin{lemma}[Generalized product formula]\label{lm:generalized-productformula}
For any $u\in K^\times$, we have
\[\sum_{\Gamma\subset \mathcal{B}^{(1)}}a_\Gamma\log|u|_\Gamma+\int_{\mathcal{B}(\mathbb{C})}\log|u(b)| \, \mathrm{d}\mu_{\bar{\mathscr{L}}}(b)=0~.\]
\end{lemma}
Note also that, for any $u\in K^\times$, the induced function $u:\mathcal{B}(\mathbb{C})\longrightarrow \mathbb{P}^1(\mathbb{C})$ is rational, whence it is holomorphic and non-vanishing on $U:=\mathcal{B}(\mathbb{C})\setminus \mathrm{supp}(\mathrm{div}(u))$, whence the function $b\in U\mapsto \log|u|_b=\log|u(b)|$ is pluriharmonic on $U$. 

\subsubsection{Generalized product formula and finite extensions}
A finite extension $L$ of the field $K$ has the same transcendence degree over $\mathbb{Q}$ as $K$. We follow here~\cite[\S~3B]{Vojta}. Let $\mathcal{B}'$ be the normalization of $\mathcal{B}$ in $L$ and let $p:\mathcal{B}'\to\mathcal{B}$ be the associated map. Then $p$ is a finite morphism of degree $[L:K]$ and $\mathscr{L}':=p^*\mathscr{L}$ is an ample line bundle on $\mathcal{B}'$ and we also have a product formula for the polarization $p^*\bar{\mathscr{L}}$ on $\mathcal{B}'$. In this case, if $v\in M_K$ and $w\in M_L$, we say that $w$ lies over $v$ and we also denote by $w|v$ if
\begin{itemize}
\item either $w$ and $v$ are both archimedean and correspond respectively to points $b'\in \mathcal{B}'(\C)^\mathrm{gen}$ and $b\in \mathcal{B}(\C)^{\mathrm{gen}}$ with $p(b')=b$,
\item or $w$ and $v$ are non-archimedean, corresponding to prime divisors $Y'$ on $\mathcal{B}'$ and $Y$ on $\mathcal{B}$ respectively, and $p(Y')=Y$.
\end{itemize}

\subsection{The proof of Theorem~\ref{tm:Gal-Fekete} over a number field} Of course, when $f$ is defined over a number field, it is also defined over a finitely generated field over $\mathbb{Q}$ containing it. Nevertheless, the number field is as a toy-model where the ideas are more transparent and we also want to keep track of the dependency on the field of the different constants. 

So, assume first $f$ is defined over a number field and let $F$ be the good lift of $f$ (i.e.~ $\mathrm{Res}(F)=1$). Then $F$ is defined over a number field $\mathbb{K}$ and let $a_d,\ldots,a_0,b_d,\ldots,b_0\in \mathbb{K}$ be its coefficients. Write $E=\{z_1,\ldots,z_N\}$ with $N:=\# E$. Up to replacing $\mathbb{L}:=\mathbb{K}(E)$ by a finite extension, we can assume it is a Galois extension. By assumption, the set $E$ is invariant under the action of the finite group $G=\mathrm{Gal}(\mathbb{L}/\mathbb{K})$.

As all points in $E$ are preperiodic under iteration of $f$ and as $E$ is finite, there are integers $n>m\geq0$ such that $f^{\circ n}(z_i)=f^{\circ m}(z_i)$ for all $1\leq i\leq N$. Up to replacing $\mathbb{L}$ by a finite extension, we thus can find lifts $Z_1,\ldots,Z_N\in \mathbb{L}^2\setminus\{(0,0)\}$ of $z_1,\ldots,z_N$ respectively such that $F^{\circ n}(Z_i)=F^{\circ m}(Z_i)$ for $1\leq i\leq N$. In particular, for any absolute $w\in M_\mathbb{L}$ and any $i$, we have $G_{F,w}(Z_i)=0$ and
\[\sum_{1\leq i\neq j\leq N}\Phi_{g_f,w}(z_i,z_j)=\sum_{1\leq i\neq j\leq N}\log|Z_i\wedge Z_j|_w=\log\left|\Delta_E\right|_w,\]
where $\Delta_E=\prod_{1\leq i\neq j\leq N}Z_i\wedge Z_j$. 
We now let $v_0\in M_\mathbb{K}$ be the archimedean place of $\mathbb{K}$ corresponding to the our initial choice of embedding $\mathbb{K}\hookrightarrow \C$. As $v_0$ is archimedean, for any $w\in M_{\mathbb{L}}$ with $w|v_0$, we have $[\C_w:\C_v]=[\C:\C]=1$. Since $E$ is $G$-invariant, the definition of $\Delta_E$ implies
\[|\sigma(\Delta_E)|_w=|\Delta_E|_w\]
for any $\sigma\in G$ and any $w\in M_\mathbb{L}$ with $w|v_0$.
By the product formula~\eqref{productformulanbfield2} , we have
\begin{align*}
-[\mathbb{L}:\mathbb{K}]\sum_{1\leq i\neq j\leq N}\Phi_{g_f}(z_i,z_j) & =
-\sum_{w_0|v_0}\log|\Delta|_{w_0}&\\
& =\sum_{w\in M_\mathbb{L} \ w \nmid v_0}[\mathbb{L}_w:\mathbb{K}_v]N_v\log|\Delta|_{w}\\
&= \sum_{w\in M_\mathbb{L} \ w \nmid v_0}[\mathbb{L}_w:\mathbb{K}_v]N_v\sum_{1\leq i\neq j\leq N}\Phi_{g_f,w}(z_i,z_j).
\end{align*}
We now use Lemma~\ref{lm:BakerHsia}. We get
\begin{align*}
-\sum_{1\leq i\neq j\leq N}\Phi_{g_f}(z_i,z_j) 
 &\leq \sum_{w\in M_\mathbb{L} \ w \nmid v_0}\frac{[\mathbb{L}_w:\mathbb{K}_v]}{[\mathbb{L}:\mathbb{K}]}N_v C_d\left(\epsilon_w+\log^+\max\{|a_i|_w,|b_j|_w\}\right)N\log(N)
\end{align*}
and this sum is finite, $\epsilon_w=0$ if $w$ is non archimedean and $\epsilon_w=1$ if $w$ is archimedean. As $a_0,\ldots,a_d,b_0,\ldots,b_d\in \mathbb{K}$, for a given $v\in M_\mathbb{K}$, we have
\[\sum_{w |  v}\frac{[\mathbb{L}_w:\mathbb{K}_v]}{[\mathbb{L}:\mathbb{K}]}N_v \log^+\max\{|a_i|_w,|b_j|_w\}=N_v\log^+\max\{|a_i|_v,|b_j|_v\}.\]
An easy computation also gives
\[\sum_{w\in M_\mathbb{L}^\infty \ w \nmid v_0}\frac{[\mathbb{L}_w:\mathbb{K}_v]}{[\mathbb{L}:\mathbb{K}]}N_v C_d\epsilon_w\leq \sum_{w\in M_\mathbb{L}^\infty }\frac{[\mathbb{L}_w:\mathbb{K}_v]}{[\mathbb{L}:\mathbb{K}]}N_v C_d\epsilon_w= C_d[\mathbb{K}:\mathbb{Q}].\]
The combination of the three last equations above gives
\begin{align*}
-\sum_{1\leq i\neq j\leq N}\Phi_{g_f}(z_i,z_j) 
 &\leq C_d[\mathbb{K}:\mathbb{Q}]\left(1+h_{\mathrm{nv}}([a_i:b_j])\right)N\log(N),
\end{align*}
where $h_{\mathrm{nv}}([a_i:b_j])= \sum_{v\in M_\mathbb{K}}\frac{N_v}{[\mathbb{K}:\mathbb{Q}]}\log^+\max\{|a_i|_v,|b_j|_v\}$ is the absolute naive height of $[a_i:b_j]$.
This is the expected inequality.

\subsection{The proof of Theorem~\ref{tm:Gal-Fekete} in the general case}
We now let $f$ be a degree $d$ rational map and we let $F$ be its good lift. As above, we let $a_d,\ldots,a_0,b_d,\ldots,b_0$ be its coefficients. Then, $f$ and $F$ are defined over $K:=\mathbb{Q}(a_d,\ldots,a_0,b_d,\ldots,b_0)$. We assume $K$ is not a number field so it is a finitely generated field over $\mathbb{Q}$ with $\mathrm{trdeg}(K/\mathbb{Q})\geq1$. By section~\ref{Sec:finitelygenerated}, there is  a normal projective variety $\mathcal{B}$ defined over $\mathrm{Spec}(\mathbb{Z})$ with function field $\mathbb{Q}(\mathcal{B})=K$. Let $\mathscr{L}$ be an ample line bundle on $\mathcal{B}$. We are going to use generalized product formulas on the field $K$ induced by hermitian continuous metrizations on the pair $(\mathcal{B},\mathscr{L})$ as in the Section~\ref{Sec:finitelygenerated}.
As described in \S~\ref{Sec:finitelygenerated}, the set $M_K$ is the disjoint union
\[M_K=\mathcal{B}^{\mathrm{gen}}(\mathbb{C}) \sqcup \bigsqcup_{\Gamma\in\mathcal{B}^{(1)}}\{|\cdot|_\Gamma\}\]
and for any $u\in K^\times$, the set $\{\Gamma\in\mathcal{B}^{(1)}\, ; |u|_\Gamma\neq1\}$ is a finite set. Note that the initial embedding $K\hookrightarrow \mathbb{C}$ corresponds to a point $b_0\in \mathcal{B}^{\mathrm{gen}}(\mathbb{C})$.

\medskip

As above, let $L$ be a finite Galois extension of $K$ such that $E\subset \p^1(L)$ and we can assume we have lifts $Z_1,\ldots,Z_N\in L^2$ of $z_1,\ldots,z_N$ respectively that are preperiodic under $F$. Let $p:\mathcal{B}'\to\mathcal{B}$ be the finite morphism induced by the field extension $L/K$ and $\mathscr{L}'=p^*\mathscr{L}$. Again, for any $w\in M_L$, one has
\[\sum_{1\leq i\neq j\leq N}\Phi_{g_f,w}(z_i,z_j)=\sum_{1\leq i\neq j\leq N}\log|Z_i\wedge Z_j|_w=\log\left|\Delta_E\right|_w,\]
where $\Delta_E=\prod_{1\leq i\neq j\leq N}Z_i\wedge Z_j$. Denote $G=\mathrm{Gal}(L/K)$, and let $b\in \mathcal{B}'(\C)^{\mathrm{gen}}$. As the set $E$ is $G$-invariant, the above gives
\[|\Delta_E(\sigma\cdot b)|=|\sigma(\Delta_E)|_{b}=|\Delta_E|_{b}=|\Delta_E(b)|\]
for any $\sigma\in G$. This implies $|\Delta_E(b)|=|u\circ p(b)|$ for all $b\in \mathcal{B}'(\C)^\mathrm{gen}$ for some rational function $u\in \mathbb{Q}(\mathcal{B})$. As $\deg(p)=[L:K]$, this gives in particular
\begin{align}
[L:K]\log|u(b_0)|=\sum_{p(b')=b_0}\log|\Delta_E|_{b'}=[L:K]\sum_{1\leq i\neq j\leq N}\Phi_{g_f}(z_i,z_j).\label{eq:reduce-to-b0}
\end{align}
Let $S:=\{\Gamma\in \mathcal{B}^{(1)}\, : \ \max\{|a_i|_\Gamma,|b_j|_\Gamma\}>1\}$ and $P:=\{p$ prime $; \ \Gamma\in S\cap \mathcal{B}^{(1)}_p\neq\varnothing\}$. Note that the sets $S$ and $P$ are finite and independent of $E$. By Lemma~\ref{lm:BakerHsia}, for any $\Gamma\in {\mathcal{B}'}^{(1)}$,
\[\log|\Delta_E|_\Gamma\leq C_d\log^+\max\{|a_i|_{p(\Gamma)},|b_j|_{p(\Gamma)}\}N\log(N)\]
for some constant $C_d$ depending only on $d$. In particular, $\log|\Delta_E|_\Gamma\leq0$ for all $\Gamma\notin p^{-1}(S)$.

\medskip

We now endow $\mathscr{L}$ with a continuous hermitian metric (which depends on $E$). We denote by $\bar{\mathscr{L}}$ the induced hermitian line bundle. We choose $\bar{\mathscr{L}}$ so that the induced measure $\mu_{\bar{\mathscr{L}}}$ is supported in $\mathcal{B}(\mathbb{C}))\setminus\mathrm{supp}(\mathrm{div}(u))$ and such that, in a local chart centered at $b_0$, the measure $\mu_{\bar{\mathscr{L}}}$ is exactly $\deg_\mathscr{L}(\mathcal{B})$ times the normalized Lebesgue measure on the Shilov boundary of a polydisk centered at $b_0$. We let $a_{\Gamma}$ be the coefficient induced by the hermitian line bundle $p^*\bar{\mathscr{L}}$ on $\mathcal{B}'$.

As $b\mapsto -\log|u(b)|$ is pluriharmonic on $\mathcal{B}(\mathbb{C})\setminus \mathrm{supp}(\mathrm{div}(u))$, the mean equality gives
\[-\log|u(b_0)| = -\frac{1}{\deg_{\mathscr{L}}(\mathcal{B})}\int_{\mathcal{B}(\mathbb{C})}\log|u(b)|\, \mathrm{d}\mu_{\bar{\mathscr{L}}}(b).\]
As $|\Delta_E|=|u\circ p|$ as functions on $\mathcal{B}'(\C)$ and as $\deg(p)\deg_{\mathscr{L}}(\mathcal{B})=\deg_{\mathscr{L}'}(\mathcal{B}')$, the projection formula gives
\[-\log|u(b_0)| = -\frac{1}{\deg_{\mathscr{L}'}(\mathcal{B}')}\int_{\mathcal{B}'(\mathbb{C})}\log|\Delta_E(b')|\, \mathrm{d}\mu_{p^*\bar{\mathscr{L}}}(b').\]
We now apply the generalized product formula given by Lemma~\ref{lm:generalized-productformula} to find
\begin{align*}
-\log|u(b_0)|
& =\sum_{\Gamma\in (\mathcal{B}')^{(1)}}\frac{a_{\Gamma'}}{\deg_{\mathscr{L}'}(\mathcal{B}')}\log|\Delta_E|_\Gamma\\
& \leq C_d\sum_{p(\Gamma)\in S}\frac{a_{\Gamma}}{\deg_{\mathscr{L}'}(\mathcal{B}')}\log^+\max\{|a_i|_{p(\Gamma)},|b_j|_{p(\Gamma)}\}\cdot N\log(N).
\end{align*}
We now use \eqref{eq:finiteplace}: for a given prime number $p$, $\sum_{\Gamma\in {\mathcal{B}'}^{(1)}_p}a_\Gamma\leq \deg_{\mathscr{L}'}(\mathcal{B}')$ and the above gives
\begin{align*}
-\log|u(b_0)|
& \leq C_d\sum_{p\in P}\sum_{\Gamma\in \mathcal{B}^{(1)}_p}\log^+\max\{|a_i|_{\Gamma},|b_j|_{\Gamma}\}\cdot N\log(N).
\end{align*}
Let $C:=C_d\sum_{p\in P}\sum_{\Gamma\in \mathcal{B}^{(1)}_p}\log^+\max\{|a_i|_{\Gamma},|b_j|_{\Gamma}\}<+\infty$. Together with \eqref{eq:reduce-to-b0}, this concludes the proof.

\subsection{Proof of Theorems~\ref{Main-theorem} and \ref{tm:Galois-quantitative}}
\begin{proof}[Proof of Theorem~\ref{tm:Galois-quantitative}] It is a direct consequence of Theorems~\ref{tm:speed-quasi-Fekete} and \ref{tm:Gal-Fekete} since Lipschitz observables are in $W^{1,2}$ with $\|\nabla \varphi\|_{L^2}\leq \mathrm{Lip}(\varphi)$.
\end{proof}

\begin{proof}[Proof of Theorem~\ref{Main-theorem}] 
By a classical interpolation argument, it is sufficient to prove the result for Lipschitz  observables.
Let $f:\p^1\to\p^1$ be a degree $d$ rational map.
To prove Theorem~\ref{Main-theorem}, it is sufficient to show that Theorem~\ref{tm:Galois-quantitative} applies to the set $\mathrm{Per}_n(f)$ and also that there is a constant $0<\alpha<1$ such that 
\[\alpha d^n\leq \# \mathrm{Per}_n(f)\leq d^n, \quad n\geq 2.\]
Let $F$ be the good lift of $f$ and $K$ be the field generated by its coefficients. Then $\mathrm{Per}_n(f)$ is the set of zeros of the dynatomic polynomial $\Psi_{n,f}$, which has coefficients in $K$. Whence $\mathrm{Per}_n(f)$ is a finite $\mathrm{Gal}(\bar{K}/K)$-invariant subset of $\p^1(\C)$. Theorem~\ref{tm:Gal-Fekete} thus applies to $\mathrm{Per}_n(f)$.

By definition, we have $\# \mathrm{Per}_n(f)\leq d^n$ and we know (\cite{lyubich:equi}) that the number of repelling periodic points of period exact $n$ is $\sim d^n$. This conclude the proof. 
\end{proof}
\begin{remark} \normalfont
	For the counting of $\# \mathrm{Per}_n(f)$, we could give a more quantitative bound using the combinatorial of petals of a parabolic periodic points to bound from above their number. This would rely on Fatou-Shishikura inequality which is not available in higher dimension whereas Lyubich result is by \cite{briendduval} or the improved version of \cite{DeThelin-Dinh-Kaufmann}. 
\end{remark}

\bibliographystyle{short}
\bibliography{biblio}
\end{document}